\documentclass[10pt,reqno]{amsart}
\usepackage[utf8]{inputenc}
\usepackage[english]{babel}
\usepackage{amsmath, amssymb, amsthm}
\usepackage{mathrsfs}
\usepackage[shortlabels]{enumitem}
\usepackage{graphicx}
\usepackage{tikz-cd}
\usepackage{xcolor}
\usepackage{mathtools}
\numberwithin{equation}{section} 

\usepackage{diagbox} 
\usepackage{tensor}  

\mathtoolsset{showonlyrefs=true}

\usepackage{hyperref} 

\usepackage[alphabetic]{amsrefs} 

\theoremstyle{plain}
\newtheorem{thm}{Theorem}[section]

\newtheorem{lemma}[thm]{Lemma}

\newtheorem{quest}{Question}

\theoremstyle{definition}

\newtheorem{rmk}[thm]{Remark}

\allowdisplaybreaks

\newcommand{\R}{\mathbb{R}}

\newcommand{\eps}{\varepsilon}

\DeclareMathOperator{\real}{Re}

\DeclareMathOperator{\vol}{Vol}

\newcommand{\A}{\mathbb{A}}

\DeclareMathOperator{\tr}{tr}

\def\XXint#1#2#3{{\setbox0=\hbox{$#1{#2#3}{\int}$ }
\vcenter{\hbox{$#2#3$ }}\kern-.6\wd0}}

\usepackage{scalerel,stackengine}
\stackMath
\newcommand\hhat[1]{%
\savestack{\tmpbox}{\stretchto{%
  \scaleto{%
    \scalerel*[\widthof{\ensuremath{#1}}]{\kern.1pt\mathchar"0362\kern.1pt}%
    {\rule{0ex}{\textheight}}
  }{\textheight}%
}{2.4ex}}%
\stackon[-6.9pt]{#1}{\tmpbox}%
}

\DeclareMathOperator{\diff}{Diff}

\newcommand{\defeq}{\coloneqq}
\DeclareMathOperator{\id}{Id}

\renewcommand{\l}{\left}
\renewcommand{\r}{\right}
\newcommand{\mc}{\mathcal}

\DeclareMathOperator{\Diff}{Diff}

\newcommand{\ip}[2]{\left\langle #1,#2 \right\rangle}

\DeclareMathOperator{\ran}{ran}



\title[Marked Magnetic Action Rigidity]{Marked Magnetic Action Rigidity}
\author{Louis-Brahim Beaufort}
\address{Universit\'e Paris-Saclay, Laboratoire de math\'ematiques d’Orsay, 91405, Orsay, France.}
\email{louis-brahim.beaufort@math.cnrs.fr}

\author{Sebastián Muñoz-Thon}
\address{Universit\'e Paris-Saclay, Laboratoire de math\'ematiques d’Orsay, 91405, Orsay, France.}
\email{sebastian.munoz-thon@universite-paris-saclay.fr}

\author{Sean Richardson}
\address{Department of Mathematics, University of Washington, Seattle, WA 98195.}
\email{seanhr@uw.edu}

\begin{document}

\begin{abstract}
An exact magnetic system over a closed manifold $M$ consists of a pair $(g,\alpha)$, where $g$ is a Riemannian metric and $\alpha$ is a 1-form encoding a magnetic field. In this context, we consider a generalization of the marked length rigidity conjecture: does the marked magnetic action spectrum of magnetic systems with Anosov magnetic flow determine the metric and the 1-form, up to a natural obstruction? In this article we answer this question in two settings: 1) locally for systems with close metrics and 1-forms and 2) for metrics in the same conformal class.
\end{abstract}

\maketitle


\section{Introduction}

Given a Riemannian manifold $M$ with Anosov geodesic flow, there exists a unique length-minimizing geodesic in each free homotopy class. The \emph{marked length spectrum} is the map that assigns to each free homotopy class the length of the minimizing geodesic in it. The Burns--Katok conjecture \cite{1985-burns-katok} asks if this map determines the metric $g$ up to a natural obstruction (a diffeomorphism of $M$ homotopic to the identity). 

In this article, we give partial answers to the magnetic version of this conjecture. An \emph{exact magnetic system} on $M$ consists of a pairing $(g,\alpha)$ of a Riemannian metric $g$ with a 1-form $\alpha$. We define the \emph{Lorentz force} as the map $Y \colon TM \to TM$ by
\[ 
d\alpha_{x}(v,w)=g_{x}(w,Y_{x}(v)), 
\]
where $v,w \in T_{x}M$. A \emph{magnetic geodesic} is a curve $\gamma \colon [0,T] \to M$ with acceleration governed by the Lorentz force
\begin{equation} \label{eq:mag-geo}
    \nabla_{\dot{\gamma}} \dot{\gamma}(t) = Y_{\gamma(t)}(\dot{\gamma}(t))
\end{equation}
where $\nabla$ is the covariant derivative induced by the Levi-Civita connection of $g$. Equation \eqref{eq:mag-geo} defines a flow $\varphi_{t} \colon SM \to SM$ over the unit tangent bundle $SM = \{(x,v) \in TM : |v|_g = 1\}$ called the \emph{magnetic flow}. We will work with magnetic flows that are Anosov (see Section \ref{section:prelim}). Given a smooth curve $\gamma \colon [0,T] \to M$, the \emph{(time-free) action} with respect to the magnetic system $(g,\alpha)$ is given by
\[ 
\A(\gamma) = \A_{g, \alpha}(\gamma) = \frac{1}{2}\int|\dot{\gamma}|_{g}^{2}+\frac{T}{2}-\int_{\gamma}\alpha.
\]
This is a natural action to consider as it is the action used in Ma\~n\'e's action potential at critical energy level $1/2$, see \cite{M97, CIPP00, BK02}.
Importantly, in every non-trivial free homotopy class $c$, there is a unique minimizer of $\A$, which will be a unit-speed magnetic geodesic $\gamma_{g,\alpha}(c)$ \cite{CIPP00}*{Theorem 29}. It is a quick exercise to verify that in the case of an everywhere vanishing magnetic field, these minimizing values are precisely the lengths of the closed geodesics. Thus, in similar fashion to how \cite{DPSU07} considers a generalization of the boundary rigidity problem using this action, we consider the following natural generalization of the marked length spectrum to the magnetic setting. Letting $\mathcal{C}$ denote the set of free homotopy classes on $M$, we define the \emph{marked magnetic action} as the map
\begin{align*}
    \mathcal{C} \to \R,
    \quad\quad c \mapsto \A_{g,\alpha}(\gamma_{g,\alpha}(c)).
\end{align*}
It is not hard to see that $\A$ is invariant under the action by pullback of a diffeomorphism homotopic to the identity, and the addition of an exact 1-form to $\alpha$, thus $\A_{g,\alpha}(c)=\A_{\phi^{*}g,\phi^{*}\alpha+d\varphi}(c)$. Hence, as a natural generalization of the marked length spectrum rigidity conjecture, we consider the following:

\begin{quest} \label{quest:mar}
    Does $(\A(\gamma(c)))_{c \in \mathcal{C}}$ determine $(g,\alpha)$ up to a diffeomorphism homotopic to the identity and an exact closed form?
\end{quest}

In this article, we answer Question \ref{quest:mar} in two cases. First, we consider magnetic systems in a small neighborhood of a fixed system with solenoidal-injective magnetic X-ray transform $I_2$ (see Section \ref{sec:magnetic_flows}). Using a stability estimate obtained in \cite{2025-munoz-thon_Richardson}, we are able to provide the following local result analogous to the main result of \cite{2019-guillarmou-lefeuvre} for the geodesic case.

\begin{thm} \label{thm:local_full}
Let $(M,g_{0},\alpha_{0})$ be a $C^{k}$-Anosov magnetic system, whose magnetic X-ray transform $I_{2}$ is solenoidal-injective. There exist $N,\eps>0$ such that the following holds. Let $(g,\alpha)$ be a magnetic system with the same marked magnetic action as $(g_{0},\alpha_0)$, and such that 
   \[ 
   \|g-g_{0}\|_{C^{N}(M)}+\|\alpha-\alpha_{0}\|_{C^{N}(M)} \leq \eps.
   \]
   Then there exists a diffeomorphism $\phi \colon M \to M$ isotopic to the identity, and an exact form $d\varphi$ such that 
   \[ (g,\alpha)=(\phi^{*}g_{0},\phi^{*}\alpha_{0}+d\varphi). \]
\end{thm}

A key difficulty in generalizing from the geodesic flow to the magnetic case is that the linearization argument requires a coupling $[g, \alpha]$ of a symmetric 2-tensor with a 1-tensor. Such a coupling is also required for the magnetic versions of the potential, divergence, and normal operators, so we have a coupling of Fourier modes throughout our argument.

We now discuss sufficient conditions for solenoidal-injectivity of the magnetic X-ray transform acting on symmetric tensors of order $\leq 2$. It was shown in \cite{Ainsworth15} that it is injective on surfaces. In higher dimensions, two similar criteria exist in the literature. The first one is obtained in \cite{DPSU07} on manifolds with boundary, and was translated recently to closed manifolds in \cite{2025-munoz-thon_Richardson}. Denote, for $(x, v) \in SM$,
\begin{align*}
k_{x}(v) = \sup_{z \in S_x M, z \perp v} 2 \mathcal{R}(z, v, v, z) + g(Y_x(v), z )^2 + (n+3) | Y_x z |_{g}^2 - 2 g((\nabla_x Y_p) (v), z ),
\end{align*}
where $\mathcal{R}$ the Riemann curvature tensor. Then, if for any closed orbit $\gamma$ with period $T_\gamma$, we have 
\begin{align} \label{eq:crit_dp}
T_\gamma \int_0^{T_\gamma} \max(0, k(\gamma(t), \dot{\gamma}(t)) dt \leq 4,
\end{align}
it follows that the magnetic X-ray transform is solenoidal injective on tensors of order $\leq 2$.

The second criterion was obtained in \cite{Beaufort-magnetic-tom} and may be expressed in terms of the magnetic sectional curvature $\mathrm{sec}^{g, \alpha}$ recently introduced in \cite{Assenza24}, which is defined as follows. For $v, z \in S_x M$, $x \in M$,
\begin{align} \label{eq:def_mag_sec}
\mathrm{sec}^{g, \alpha}_v(w) = \mathcal{R}(w, v, v, w) + g( (\nabla_w Y_x) (v), w ) + \frac{1}{4} \vert Y_x (w) \vert_g^2 + \frac{3}{4} g( w, Y_x (v) )^2.
\end{align}
Then, if we have 
\begin{align} \label{eq:crit_b}
    \mathrm{sec}^{g, \alpha}_v(w) + \left ( \frac{n}{2} - 1 + \frac{2}{n+2} \right )g( w, Y_x (v) )^2 < 0, \quad \forall x \in M, \, \forall v, w \in S_x M,
\end{align}
it follows that the magnetic X-ray transform is solenoidal injective on tensors of order $\leq 2$. 

We note both conditions are satisfied for negatively curved $(M, g)$ with $\alpha$ sufficiently small in $C^2$-norm (for \eqref{eq:crit_b} $C^{1}$-small norm is enough). Then, Theorem \ref{thm:local_full} implies:

\begin{thm} \label{thm:local}
   Let $(M,g_{0},\alpha_{0})$ be an Anosov magnetic system. Assume either
   \begin{enumerate}
       \item $M$ is a surface.
       \item $\dim M \geq 3$ and either \eqref{eq:crit_dp} or \eqref{eq:crit_b} holds.
   \end{enumerate}
   There exists $N,\eps>0$ such that the following holds. Let $(g,\alpha)$ be a magnetic system with the same marked magnetic action as $(g_{0},\alpha)$, and such that 
   \[ 
   \|g-g_{0}\|_{C^{N}(M)}+\|\alpha-\alpha_{0}\|_{C^{N}(M)} \leq \eps.
   \]
   Then there exists a diffeomorphism $\phi \colon M \to M$ isotopic to the identity, and a exact form $d\varphi$ such that 
   \[ (g,\alpha)=(\phi^{*}g_{0},\phi^{*}\alpha_{0}+d\varphi). \]
\end{thm}

Just as in \cite{2019-guillarmou-lefeuvre}, we also obtain a stability result for the magnetic action.

\begin{thm} \label{thm:stability}
Let $(M,g_{0},\alpha_{0})$ be as in Theorem \ref{thm:local_full}. For $s>0$ small enough, there exists constants $C,N,\tau>0$ such that the following holds: there exists $\eps>0$ small enough such that for any magnetic system $(g,\alpha)$ $\eps$-close in the $C^{N}$-topology to $(g_{0},\alpha_{0})$, there exists a diffeomorphism $\phi$ close to $\mathrm{Id}$ such that 
\[
\|[\phi^{*}g,\phi^{*}\alpha+d\varphi]-[g_{0},\alpha_{0}] \|_{H^{s}}\leq C\|\A_{g,\alpha}-\A_{g_{0},\alpha_{0}}\|_{\ell^{\infty}(\mathcal{C})}^{\tau}\|[g,\alpha]-[g_{0},\alpha_{0}] \|_{C^{N}}^{1-\tau}.
\]
\end{thm}

\begin{rmk}
Although we do not write explicitly the regularity on our previous results, we need the magnetic system to be at least of class $C^{4}$ in order to apply Lemma \ref{lemma:taylor}. We also remark that one can obtain an explicit $N$ by following the proofs of the results.    
\end{rmk}

Our second rigidity result deals with magnetic systems in a given conformal class. In this case, we obtain a full answer. 

\begin{thm} \label{thm:conformal}
    Let $(g_{1},\alpha_{1})$ $(g_{2},\alpha_{2})$ be Anosov magnetic systems with the same marked action spectrum. If $g_{1}$ is conformal to $g_{2}$, i.e., if there exists $f \in C^{\infty}(M)$ so that $g_{2}=e^{2f}g_{1}$, then $f \equiv 0$ and $\alpha_{2}=\alpha_{1}+d\varphi$, for some $\varphi \in C^{\infty}(M)$. 
\end{thm}

\subsection{Perspectives and related results}

The Burns--Katok conjecture \cite{1985-burns-katok} has been solved in several cases, although is wide open in higher dimensions. To begin with, Katok \cite{Katok_conformal} solved the conjecture on a fixed conformal class. In the case of negatively curved surfaces, a solution was given by Croke \cite{Croke-rigidity-2d} and Otal \cite{Otal-rigidity-2d}. A deformation result was given by Croke and Sharafutdinov as a consequence of the injectivity of the X-ray transform for manifolds of negative curvature \cite{CS98}. In the case of locally symmetric spaces, the conjecture was solved by Hamenstädt \cite{Hamenstadt-rigidity-LS}. After twenty years, there was a breakthrough by Guillarmou and Lefeuvre \cite{2019-guillarmou-lefeuvre}, showing the results for metric that are close enough. The result on surface was recently generalized to the case of Anosov metrics in \cite{GLP25}, and has also been studied recently for both magnetic and other dissipative flows \cite{EC25}. We point out that the result of \cite{EC25} could hypothetically be generalized along the lines of \cite{GLP25} to obtain a rigidity result on the length of magnetic orbit, but not the magnetic action which we consider here.

    In the case of magnetic system on closed manifolds, a deformation result was obtained on surfaces by Marshall Reber \cite{MR-deformative}. This was recently generalized in \cite{AdSMRT-magnetic-rigidity-2d}, where rigidity was proved using the length of magnetic geodesics in place of the action, we discuss this below. The study of the linearized problem i.e., of the magnetic ray transform, that is, the operator that integrates functions (or pair of tensors) over magnetic geodesic (see Section \ref{sec:magnetic_flows} for the precise definition) has been studied for 0- and 1-tensors in \cite{DP05} and \cite{DP07} on surfaces, and in \cite{DP08} in the higher dimension case. For 2-tensors, in the case of surfaces, it was studied in \cite{Ainsworth15}. The last two authors studied the case of 2-tensors in \cite{2025-munoz-thon_Richardson} under a bound in the sectional curvature and the Lorentz force (which holds, for example, for negatively curved manifolds with small magnetic 1-form in the $C^{2}$-norm). In the forthcoming paper \cite{Beaufort-magnetic-tom}, the first author solves the problem for general $m$-tensors, under a bound in the magnetic sectional curvature (defined in \eqref{eq:def_mag_sec}). 

Regarding our results, Theorems \ref{thm:local_full} and \ref{thm:local} are motivated by the local rigidity result in \cite{2019-guillarmou-lefeuvre}. In the same way as there, we show that when a magnetic system is close to a given one, we can deform it to a one that is divergence free (or solenoidal, see Section \ref{sec:magnetic_flows}). We then show that the potential part of the difference vanishes. This last step uses, in a implicit way, microlocal techniques. Indeed, we use stability results for the linearized problem obtained recently by the last two authors \cite{2025-munoz-thon_Richardson}. Although our bounds follow the spirit from \cite{2019-guillarmou-lefeuvre}, we avoid using the unstable Jacobian, and we neither use the geodesic stretch as in the proof proposed later in \cite{GKL22} (see also \cite{Lefeuvre_book}*{Chapter 17}). Hence, our result also gives a new proof of the local rigidity result by Guillarmou--Lefeuvre by using the stability of the geodesic X-ray transform on closed manifolds. The hypothesis of Theorem \ref{thm:local} regarding the dimension of $M$ or the magnetic sectional curvature, are due to the fact that the stability result from \cite{2025-munoz-thon_Richardson} only holds if the linearized problem is solved (i.e., when the magnetic ray transform is solenoidal injective, see Section \ref{sec:magnetic_flows} for the proper definitions). Our stability result, Theorem \ref{thm:stability} is inspired in the analog result for the marked length spectrum from \cite{2019-guillarmou-lefeuvre}, but it can also be related to the approximate and quantitative marked length spectrum results by Butt \cite{Butt_approx_mls}, \cite{Butt_quantitative_mls}. As in \cite{2019-guillarmou-lefeuvre}, the proof Theorem \ref{thm:stability} follows a deformation argument together with the smoothness of the action (presented in Section \ref{section:lemmata}), the stability of the magnetic X-ray transform (Lemma \ref{lemma:stab}), and Sobolev estimates.

Our rigidity result in a same conformal class is inspired by the corresponding result in the Riemannian case by Katok \cite{Katok_conformal}, and the corresponding result in the case of magnetic systems on manifolds with boundary \cite{DPSU07}*{Theorem 6.1}. The proof of this last result uses Santal\'o's formula which is not suitable for closed manifolds. Instead, we argue as in \cite{Katok_conformal} and use the fact that the magnetic flow preserves the Liouville measure, which implies that this measure can be approximated by Dirac masses over closed trajectories of the flow. This together with a H\"older inequality argument, show that the conformal factor is trivial. The relation of the magnetic 1-forms then follows from algebraic manipulations together with an application of the Non-negative Liv\v sic theorem, and the fact that the magnetic X-ray transform acting on 1-forms is injective \cite{DP08}.

Finally, let us mention why we consider the action instead of the length as in \cite{MR-deformative}, \cite{AdSMRT-magnetic-rigidity-2d}. Although the length of closed orbits is more natural from the dynamical point of view, from the physical and analytic one the action is more appropriate. Indeed, magnetic geodesics are minimizers of the action but not of the length. Furthermore, since we use the microlocal methods from \cite{Guillarmou17} and \cite{2019-guillarmou-lefeuvre} to prove the local result, one has to study first the linearized problem. In the case of the geodesic flow, the linearization of the length gives the geodesic X-ray transform, although in the case of magnetic systems, linearizing the lengths of magnetic geodesics transforms gives birth to an operator that is not of X-ray type. However, the linearization of the action is indeed the magnetic X-ray transform acting on 2- and 1-tensors, a operator that was studied recently by the last two authors \cite{2025-munoz-thon_Richardson}, generalizing \cite{Guillarmou17}. Finally, we also mention that the rigidity results in the case of manifolds with boundary (which also study the magnetic X-ray transform) \cite{DPSU07} also use the action instead of the length, which also served us for a justification to tackle the action rigidity problem instead of the one with the length.

\subsection*{Organization of this article} In Section \ref{section:prelim} we introduce the machinery required for this article. In particular, we recall the definition of anisotropic Sobolev spaces associated to Anosov flows, which serve as the natural ambient where Guillarmou's $\Pi$ operator is defined. There we also introduce the magnetic X-ray transform and its associated normal operator defined in terms of $\Pi$. In Section \ref{section:lemmata} we prove some technical results that will be used in the proof of our local result. Finally, Theorems \ref{thm:local_full}, \ref{thm:stability}, and \ref{thm:conformal} are proved in Sections \ref{red:section:proof_local}, \ref{sec:proof_stability}, and \ref{section:proof_conformal}, respectively.

\subsection*{Acknowledgments} The authors thank Thibault Lefeuvre and Gabriel Paternian for helpful discussions. L.-B.B and S.M-T. were supported by the European Research Council (ERC) under the European Union’s Horizon 2020 research and innovation programme (Grant agreement no. 101162990 -- ADG).

\section{Preliminaries} \label{section:prelim}

\subsection{Anosov flows and Anisotropic Sobolev Spaces}

Let $\mathcal{M}$ a closed manifold, and let $\varphi_{t}$ be an Anosov flow on $\mathcal{M}$ preserving a smooth measure $d\mu$. Let $F$ be the generator of $\varphi$. Then, have a continuous flow-invariant splitting
\[ T\mathcal{M}=\R F \oplus E_{s} \oplus E_{u}, \]
and uniform constants $C,\lambda >0$ such that for all $t \ge 0$
\begin{equation} \label{eq:contr_exp}
    |d\varphi_{t}(v)|\leq Ce^{-\lambda t}|v|, \quad \forall v \in E_{s}; \qquad |d\varphi_{-t}(v)| \leq Ce^{-\lambda t}|v|, \quad \forall v \in E_{u}.
\end{equation}

Since the cotangent bundle is more natural from the point of view of microlocal analysis, we also present the dual splitting
\[ T^{*} \mathcal{M}=\R E_{0} \oplus E_{s}^{*} \oplus E_{u}^{*}, \]
where $E_{0}^{*}(E_{s} \oplus E_{u})=0$, $E_{s}^{*}(E_{s} \oplus \R F)=0$, $E_{u}^{*}(E_{u} \oplus \R F)=0$; we have similar expanding/contracting properties as in \eqref{eq:contr_exp} but with $d\varphi$ replaced by $d\varphi^{-\top}$. 

Of course, our main examples will be magnetic flows, that is, given a magnetic system $(g,\alpha)$ on $M$, we consider $\mathcal{M}=SM$ together with the magnetic flow $\varphi_{t} \colon SM \to SM$.

We will use the microlocal version of anisotropic Sobolev spaces developed in \cite{2011-faure-sjostrand}. Let $H$ be the Hamiltonian vector field on $T^{*}\mathcal{M}$ induced by the Hamiltonian $\sigma_{P}(x,\xi):=\langle\xi, F(x)\rangle$, $P:=\frac{1}{i} F$ its principal symbol, and let $(\Phi_{t})_{t \in \mathbb{R}}$ be the symplectic flow generated by $H$. Fix an \emph{order function} $m \in C^{\infty}(T^{*}\mathcal{M},[-1,1])$, i.e., $m$ is such that is 0-homogeneous in the $\xi$-variable for $|\xi|$ large enough, such that $m \equiv 1$ in a conic neighborhood of $E_{s}^{*}$, and $m \equiv -1$ in a conic neighborhood of $E_{u}^{*}$. Consider the associated \emph{escape function} $G_{m}(x,\xi):=m(x,\xi) \log|\xi|_{g}$. Then:
\begin{enumerate}[label=\roman*)]
    \item $H G_{m}(x, \xi) \leq 0$  for $|\xi|_{g} \geq R$;
    \item $H G_{m}(x, \xi) \leq-C<0$ for all $(x,\xi)$ in a conic neighborhood of $E_{s}^{*}\cup E_{u}^{*}$ and such that $|\xi|_{g} \geq R$, for some constant $C>0$.
\end{enumerate}
We now define $A_{s}:=\mathrm{Op} (e^{s G_{m}})$. Next, we define the \emph{anisotropic Sobolev spaces} by 
\[
\mathcal{H}_{\pm}^s(\mathcal{M}):=A_{s}^{\mp1}(L^2(\mathcal{M})).
\]
Morally, the space $\mathcal{H}_{+}^{s}$ (resp. $\mathcal{H}_{-}^{s}$) is defined microlocally to have Sobolev regularity $H^{s}$ (resp. $H^{-s}$) in a conic neighborhood of $E_{s}^{*}$ and regularity $H^{-s}$ (resp. $H^{s}$) in a conic neighborhood of $E_u^*$. 

The normal operator associated to the linearization of $\A$ is a map between these spaces. To define it, we first recall that the resolvents
\[
R_{\pm}(z) = (\mp F-z)^{-1}:=-\int_{0}^{\infty}e^{\mp tF}e^{-tz}dt.
\]
are well defined bounded operators on $L^{2}(\mathcal{M})$ for $\real (z) \gg 0$. It was shown in \cite{2011-faure-sjostrand}  (see also \cite{Lefeuvre_book}*{Theorem 9.1.6}) that the resolvents
\[
R_{\pm}(z)=(\mp F-z)^{-1} \colon \mathcal{H}_{\pm}^{s}(\mathcal{M}) \to \mathcal{H}_{\pm}^{s}(\mathcal{M})
\]
can be extended as a meromorphic family of bounded operators on $\{\real(z)>-cs\}$, where $c>0$ is a constant depending on the flow. Furthermore, the poles do not depend on the choices made during the construction of $\mathcal{H}_{\pm}^{s}(\mathcal{M})$. When the flow preserves a volume form, we have the Laurent expansions around 0:
\[
R_{\pm}(z)=R_{\pm}^{\mathrm{hol}}(z)-\frac{\Pi_{0}^{\pm}}{z},
\]
where the operators $\Pi^{\pm}_{0}$ are the spectral projectors
\[ 
\Pi_{0}^{\pm}=-\frac{1}{2\pi i}\oint_{\gamma} R^{\pm}(z) dz,
\]
where $\gamma$ is a curve enclosing $z=0$. \emph{Guillarmou's $\Pi$ operator} (sometimes called \emph{Guillarmou's covariance operator}, see \cite{GKL22}*{Equation (2.6)}) is then defined by
\[ 
\Pi=-(R_{+}^{\mathrm{hol}}(0)+R_{-}^{\mathrm{hol}}(0)).
\]
We record some useful properties that will be used throughout the proof of Theorem \ref{thm:local_full}.

\begin{lemma} \label{lemma:Pi_prop}
Guillarmou's $\Pi$ operator satisfies the following properties:
    \begin{enumerate}
        \item $\Pi F=0$.
        \item $\Pi \colon H^{s} \to H^{r}$ is bounded for any $r<0<s$.
    \end{enumerate}    
\end{lemma}

Since we are going to work with 2- and 1-tensors, we recall how to identify these with functions over the sphere bundle. From now on, we focus in the case of the magnetic system $(g,\alpha)$ on $M$, and its respective magnetic flow $\varphi_{t} \colon SM \to SM$. In this work, $S^{m}T^{*}M$ will denote the vector bundle of symmetric $m$-tensors, and $C^{k}(M,S^{m}T^{*}M)$ ($H^{s}(M,S^{m}T^{*}M)$) will denote the space of $C^{k}$-sections ($H^{s}$-sections respectively) on $M$ with values in
the vector bundle $S^{m}T^{*}M$. We will also denote the space of 1-forms simply by $C^{\infty}(T^{*}M)$ instead of $C^{\infty}(M,T^{*}M)$. Recall that any element $h \in C^{\infty}(M,S^{m} T^{*}M )$ can be naturally identified with a function over $SM$ by the map $\pi_m^* \colon C^{\infty}(M,S^{m}T^{*}M) \to C^{\infty}(SM)$ defined by $(\pi^*h)(x,v) = h_x(v, \ldots, v)$. This identification map has natural adjoint 
\begin{align*}
\pi_{m*} \colon \mathcal{D}'(SM) \to \mathcal{D}'(M,S^{m}T^{*}M), \quad
\ip{\pi_{m*}f}{h}_{S^m(T^*M)} \defeq \ip{f}{\pi_{m}^*h}_{SM}
\end{align*}
given by the natural pairing with a smooth symmetric covariant $m$-tensor $h \in C^{\infty}(S^{m}T^{*}M)$ induced by the metric on $M$. Here $\mathcal{D}'$ is used to denote the space of distributions. 

\subsection{Analysis of magnetic flows} \label{sec:magnetic_flows}

In this subsection we introduce operators defined in terms of the magnetic flow. We first recall the definition of the associated X-ray transform. Let $[p,q] \in C^{0}(M,S^{m} T^{*}M) \times C^{0}(M,S^{m-1} T^{*}M )$. We define the \emph{magnetic X-ray transform} of $[p,q]$ by
\[ 
I_{m}[p,q](c)=\int_{0}^{\ell_{g}(\gamma(c))} (\pi_{m}^{*}p+\pi_{m-1}^{*}q)(\varphi_{t}(x,v))dt,
\]
where $\gamma(c)$ is the unique magnetic geodesic in $c$. When we want to emphasize that the X-ray transform corresponds to the one associated to the magnetic system $(g,\alpha)$, we will write $I_{m}:=I_{m}^{g,\alpha}$. As it is shown in \cite{DPSU07}, $I_{2}$ corresponds to the linearization of $\A$ in the direction $[2p,-q]$. 

We also recall that tensors on a Riemannian manifold $M$ can be written as a sum of potential plus a divergence free part. In the same fashion, we have a similar decomposition for pair of tensors $[p,q] \in C^{1}(M,S^{m}T^{*}M) \times C^{1}(M,S^{m-1}T^{*}M)$ when $M$ is endowed with a magnetic system. We present it for $m=1,2$ which is enough for our purposes (see \cite{2025-munoz-thon_Richardson}*{Section 4} for the general result). We define the \emph{magnetic potential} and the \emph{magnetic divergence} by
\[ 
\begin{split}
D_{\mu} &\colon C^{1}(T^* M) \times C^{1}(M) \to C^{0}(M,S^{2} T^* M ) \times C^{0}(T^* M ), \\
    D_{\mu}^{*} &\colon  C^{1}(M,S^{2} T^* M ) \times C^{1}(T^{*}M) \to C^{0}(T^{*}M) \times C^{0}(M)  ,
\end{split}
\]
\[ D_{\mu}=\begin{pmatrix}
     D & 0 \\ Y & d
    \end{pmatrix},  \qquad D_{\mu}^{*}=\begin{pmatrix}
        -\tr(\nabla \bullet) & -Y \\ 0 & -\tr(\nabla \bullet)
    \end{pmatrix}, \]
respectively. Here $D$ is the symmetrized covariant derivative, $\nabla$ is the Levi-Civita connection, $d$ is the exterior differential, and $Y$ is the Lorentz force acting on 1-tensors which is defined by $Y(\xi)_{x}(v)=\xi(Y_{x}(v))$, where $\xi \in C^{0}(T^{*}M)$, $v \in T_{x}M$.

Then, the potential-solenoidal decomposition says that any pair can be written as a the sum of a pair of tensors in the image of $D_{\mu}$ together with an element in the kernel of $D_{\mu}^{*}$.

Furthermore, the decomposition can be used to characterize the kernel of $I_{m}$. We have that $\ran D_{\mu} \subset I_{m}$, and in the case of equality, we say that $I_{m}$ is \emph{solenoidal injective}, which we always abbreviate as s-injective.

After these preliminaries, we are in good shape to state the stability result for the X-ray transform (\cite{2025-munoz-thon_Richardson}*{Theorem 1.4}). This says that we can control the divergence free part of a pair of tensors by its magnetic X-ray transform.
 
\begin{lemma}[Stability] \label{lemma:stab}
    Assume that $I_{m}$ is s-injective. There exists $s_0 \in(0,1)$ and $C, \tau>0$ such that for all $f \in C^{1}( M,S^{m}T^{*}M) \times C^{1}(M,S^{m-1}T^{*}M)$ and $s <s_{0}$:
    \[ \|\pi_{\ker D_{\mu}^{*}} f \|_{H^{s_0}} \leq C \|I_m f \|_{\ell^{\infty}(\mathcal{C})}^{\tau}\|f\|_{C^1}^{1-\tau}.\]
\end{lemma}

As was mentioned in the introduction, s-injectivity (and hence, Lemma \ref{lemma:stab}) holds when the Riemannian manifold $(M,g)$ has negative sectional curvature and $\alpha$ is $C^{2}$-small enough by the results in \cite{2025-munoz-thon_Richardson}*{Section 6}. 

We will also need the operator that plays the role of the normal operator associated to $I_{m}$. As in the case of geodesic flows, it is defined in terms of Guillarmou's $\Pi$ operator. We define the \emph{magnetic normal operator}
\[
N_{m}=\begin{pmatrix}
    \pi_{m *} (\Pi+1\otimes 1) \pi_{m}^{*} & \pi_{m *} (\Pi+1\otimes 1) \pi_{m-1}^{*} \\ \pi_{m-1 *} (\Pi+1\otimes 1) \pi_{m}^{*}  & \pi_{m-1 *} (\Pi+1\otimes 1) \pi_{m-1}^{*}
\end{pmatrix},
\]
The following is a consequence of the pseudodifferential nature of $N_{m}$ (\cite{2025-munoz-thon_Richardson}*{Theorem 1.3}). 

\begin{lemma}
    Let $(M,g,\alpha)$ be an Anosov magnetic system. Then for any $s \in \R$, $N_{m}$ is a continuous map from the space $H^{s}(M,S^{m}T^{*}M) \times H^{s}(M,S^{m-1}T^{*}M)$ to $H^{s+1}(M,S^{m}T^{*}M) \times H^{s+1}(M,S^{m-1}T^{*}M)$.
\end{lemma}

Finally, we record the following version of a theorem by Lopes and Thieullen \cite{Lopes_Thieullen_Livsic} for later use.

\begin{lemma}[Non-negative Liv\v sic Theorem] \label{lemma:non-negative_livsic}
Let $\alpha \in (0,1]$, and let $F$ be the generator of an Anosov magnetic flow associated to the magnetic system $(g_{0},\alpha_{0})$. There exist a constant $C=C(g_{0},\alpha_{0})>0$ and $r \in(0,1)$ such that for any $u \in C^{\beta}(SM)$ with non-negative integral over every $\gamma_{(g_{0},\alpha_{0})}(c)$, there exist $h \in C^{\beta r}(SM)$ and $A \in C^{\beta r}(S M)$ such that $A \geq 0$ and $u+Fh=A$. Moreover, $\|A\|_{C^{\beta r}} \leq C\|u\|_{C^{\beta}}$.
\end{lemma}

\section{Lemmata} \label{section:lemmata}

In this section we prove some technical results that will be used in the proof of our local result. We begin by showing a Taylor expansion bound involving the action and the magnetic X-ray transform. In the next result, $S_{+}^{2}T^{*}M$ denotes the bundle of Riemannian metrics on $M$.

\begin{lemma} \label{lemma:taylor}
    The action $\A_{g_{0},\alpha_{0}}$ is $C^{2}$ near $(g_{0},\alpha_{0})$ in the $C^{4}(M,S_{+}^{2}T^{*}M) \times C^{4}(T^{*}M)$ topology. In particular, there exist a neighborhood $U \subset C^{4}(M,S_{+}^{2}T^{*}M) \times C^{4}(T^{*}M)$ of $(g_{0},\alpha_{0})$, and $C=C(g_{0},\alpha_{0})$ such that for all $(g,\alpha) \in U$
    \begin{align*}
        \left\|\A_{g,\alpha}-\A_{g_{0},\alpha_{0}}-I_{2}^{g_{0},\alpha_{0}}\left[\frac{1}{2}(g-g_{0}),-(\alpha-\alpha_{0})\right]\right\|_{\ell^{\infty}(\mathcal{C})} \\
        \leq C (\|g-g_{0}\|_{C^{4}(M)}^{2}+\|\alpha-\alpha_{0}\|_{C^{4}(M)}^{2}).
    \end{align*}
\end{lemma}

\begin{proof}
This is a consequence from structural stability of Anosov flows. Following the proof of \cite{2019-guillarmou-lefeuvre}*{Proposition 2.1}, we write $\mathcal{M}=S_{g_{0}}M$, and let $F_{0}$ be the magnetic vector field induced by $(g_{0},\alpha_{0})$. From perturbation theory of Anosov flows (see \cite{de_la_Llave_Marco_Moriyon_perturbation}*{Theorem A.2} and the remarks below it), there is a neighborhood $\mathcal{V}_{F_{0}} \subset C^{2}(\mathcal{M}:T \mathcal{M})$ of $F_{0}$, such that for each $F \in \mathcal{V}_{F_0}$, there exists an orbit conjugacy $\psi_F \colon \mathcal{M} \to S_g M$ which is isotopic to the identity. Moreover, 
\[
\begin{split}
    \mathcal{V}_{F_0} &\to C(\mathcal{M}, S_g M), \\ 
    F &\mapsto \psi_F 
\end{split}
\]
is a $C^2$ function. In particular, for $c$ a free homotopy class of paths in $M$, and $\gamma_{g_0, \alpha_0}(c)$ the unique closed orbit of $F_0$ in $c$, we have $\psi_{F} \circ \gamma_{g_0, \alpha_0}$ is the unique closed orbit for $F$ in $c$ (where we identify $S_{g_0} M$ with $S_g M$ by rescaling the fiber). 

As a consequence, the action $F_{g,\alpha} \mapsto \mathbb{A}(\gamma_{g, \alpha})$ is a $C^{2}$ function of the deformation $F_{g,\alpha}$. The result now follows from Taylor expansion of the action.    
\end{proof}

The next ingredient that we need is a reduction to the case of potential tensors: one can deform magnetic systems in a neighborhood of a given one, so that the deformation is solenoidal. The proof on closed Riemannian case goes back to Ebin \cite{Ebin_slice}. For smooth magnetic systems, a proof in the case of manifolds with boundary can be found in \cite{DPSU07}*{Lemma 6.6}. We give a proof in the $C^{k,\beta}$ topology for convenience of the reader. In the next statement, $D_{\mu}^{*}$ is the magnetic divergence corresponding to $(g_{0},\alpha_{0})$, and $\mathrm{Diff}_{0}(M)$ is the set of diffeomorphism homotopic to the identity. 

\begin{lemma}[Slice Lemma] \label{lemma:slice}
   Let $k \geq 2$ be an integer and let $\beta \in (0,1)$ and assume $[g_0, \alpha_0]$ to be of regularity $C^{k+2,\beta}$. Then there exists some $\eps > 0$ such that for all $g, \alpha$ satisfying
	\[
		\|[g, \alpha] - [g_0, \alpha_0]\|_{C^{k, \beta}} < \eps,
	\]
	there exists a unique $\phi \in \Diff_0^{k+1, \beta}(M)$ close to $\id$ in the $C^{k+1, \beta}$ topology together with some $\varphi \in C^{k+1, \beta}(M)$ so that
    \[
		D^*_{\mu} \l(\l[\frac{1}{2}\phi^*g, -\phi^*\alpha-d\varphi\r]- \l[\frac{1}{2}g_0, -\alpha_0\r]\r) = 0.
	\]
    Furthermore, there is some $C > 0$ so that we have the bound
    \[
        \l\|\l[\frac{1}{2}\phi^*g, -\phi^*\alpha-d\varphi\r]- \l[\frac{1}{2}g_0, -\alpha_0\r]\r\|_{C^{k,\beta}}
        \leq C\|[g, \alpha] - [g_0, \alpha_0]\|_{C^{k, \beta}}.
    \]
    \label{lem:slice-old}
\end{lemma}

\begin{proof}
We will apply the implicit function theorem. Let $U_0 \subset C^{k+1, \beta}(M;  TM)$ be a neighborhood of the $0$-section and consider the map
\[
\begin{split}
	F\colon U_0 \times C^{k+1, \beta}(M)
	\times C^{k, \beta}(M , S^2T^*M) &\times C^{k+1, \beta}( T^*M)\\
	\to C^{k-1, \beta}(T^*M) &\times C^{k-1, \beta}(M)
\end{split}
\]
defined by
\[
    F([V, \varphi], [p, q]) = D_{\mu}^*\l(\frac{1}{2}\l[e^*_V(g_0 + p), e^*_V(-\alpha_0 + q) -d\varphi\r]\r) - D_{\mu}^*\l(\l[\frac{1}{2}g_0, -\alpha_0\r]\r)
\]
where given $V \in C^{k+1, \beta}$, the map $e_V$ denotes
\[
C^{k+1, \beta}(M, T M) \to \diff^{k+1, \beta}(M), \quad e_{V}(x):=\exp _{x}(V(x))
\]
defined using the exponential map induced by $g_{0}$. Note that for $V \in U_0$, $V \mapsto e_V$ defines a well-defined smooth diffeomorphism from a small $C^{k+1, \beta}$-neighborhood of the zero section onto a neighborhood of the identity in $\diff^{k+1, \beta}(M)$.
We want to solve, locally, the equation $F([V(p, q), \varphi(p, q)], [p, q]) = 0$. 
To apply the implicit function theorem, we must show $d_{[V, \varphi]}F([0,0], [0,0])$ is an isomorphism.
	Indeed, following the ideas of \cite{DPSU07}*{Lemma 6.6}, we compute
	\begin{align*}
		d_{[V, \varphi]}F([0,0], [0,0])([Z, \psi])
		&= D^*_{\mu}\l(\l[\frac{1}{2} \mc{L}_Z g_0, \mc{L}_Z(-\alpha_0)-d\psi\r]\r) \\
        &= D^*_{\mu}([D Z^{\flat}, -d \iota_Z \alpha_0 - \iota_Z d\alpha_0 - d\psi])\\
        &= D^*_{\mu}([D Z^{\flat}, d(-\iota_Z \alpha_0 - \psi) + Y(Z^{\flat})])\\
        &= D^*_{\mu}D_{\mu}[Z^{\flat}, -\iota_Z\alpha_0 - \psi],
	\end{align*}
    where in the second equality we used $\mc{L}_Z g_0 = 2DZ^{\flat}$ (where $\flat \colon TM \to T^{*}M$ denotes the musical isomorphism) and Cartan's magic formula, in the third the definition of $Y$ (where $Y$ is the Lorentz force corresponding to $(g_{0},\alpha_{0})$), and in the last one we used the definition of the magnetic divergence. Therefore we have the following explicit form for the differential:
	\begin{equation}
	    d_{[V, \varphi]}F([0,0], [0,0])([Z, \psi])
		= D^*_{\mu} D_{\mu}[Z^{\flat}, -\iota_Z \alpha_0 - \psi].
        \label{eq:inv_der}
	\end{equation}
where $Z \in C^{k+1,\beta}(M,TM)$, $\psi \in C^{k+1,\beta}(M)$ is orthogonal to constants, and $\flat \colon TM \to T^{*}M$ is the musical isomorphism induced by $g_{0}$. Recall \cite{2025-munoz-thon_Richardson}*{Section 4} that $\ker D_{\mu}=\{[0,c]:c \in \R\}$. Since $\psi$ is orthogonal to constants, and the map $[Z,\psi] \mapsto [Z^{\flat},\iota_{Z}(-\alpha_{0})+\psi]$ is invertible, we conclude that the operator on the right-hand side of \eqref{eq:inv_der} is invertible. Thus, by the implicit function theorem for Banach spaces, we have such a neighborhood $U$ so that $F([V(p, q), \varphi(p, q)], [p, q]) = 0$. Moreover, $[V(p, q), \varphi(p, q)]$ is the unique solution in the neighborhood. Finally, note $F$ is $C^1$ and so the bound follows from the Lipschitz property of the map.    
\end{proof}


\section{Proof of Theorem \ref{thm:local_full}} \label{red:section:proof_local}

We follow \cite{2019-guillarmou-lefeuvre}. Let $(g_{0},\alpha_{0})$ as in the statement. Fix $N \geq 3$ and let $\beta \in (0,1)$ be small. Since $(g,\alpha)$ is in a $\eps$-neighborhood of $(g_{0},\alpha_{0})$, the Slice Lemma (Lemma \ref{lemma:slice}) ensures the existence of a diffeomorphism $\phi$ and a function $\varphi$ such that for
\[
    f = \l[\frac{1}{2}\phi^*g, -\phi^*\alpha-d\varphi\r]- \l[\frac{1}{2}g_0, -\alpha_0\r],
\]
we have $D_{\mu}^{*}f=0$. We will write $f = [p, q]$ and our objective is to show $f \equiv 0$. As a first step, we apply the non-negative Liv\v sic Theorem. Indeed, consider any free homotopy class $c$, let $\gamma = \gamma_{g_0, \alpha_0}(c) \colon [0,T] \to M$ be the unique magnetic geodesic in the homotopy class, and compute
\begin{align*}
    I_{2}^{g_{0},\alpha_{0}} f (c)=\l(\int_{\gamma} \frac{1}{2}\pi_2^* g +\frac{T}{2}- \int_{\gamma}\pi_1^* \alpha\r)
    - \l(\int_{\gamma} \frac{1}{2}\pi_2^* g_0+\frac{T}{2} - \int_{\gamma}\pi_1^* \alpha_0\r).
\end{align*}
Now observe the last term on the right-hand side is precisely the magnetic action $\A_{g_0, \alpha_0}(c)$. The first term on the right-hand side is $\A_{\phi^* g, \phi^{*}\alpha+d\varphi}(\gamma)$, which can be bounded from below by $\A_{\phi^{*}g, \phi^{*}\alpha+d\varphi}(c)$ since $\gamma$ does not necessarily minimize the action with respect to $(\phi^{*}g,\phi^{*}\alpha+d\varphi)$. Since $\A_{\phi^{*}g, \phi^{*}\alpha+d\varphi}(c)=\A_{g, \alpha}(c)$, we conclude that
\begin{align*}
    I_{2}^{g_{0},\alpha_{0}}(c)
    \geq \A_{g, \alpha}(c) - \A_{g_0, \alpha_0}(c) = 0,
\end{align*}
where we used our central assumption that the marked magnetic action spectrum of these two magnetic systems are equal.
Hence, by the non-negative Liv\v sic Theorem (Lemma \ref{lemma:non-negative_livsic}), there exists $h,A \in C^{r}(SM)$ with $r \in (0,1)$ such that $\pi_{2}^{*}p+\pi_{1}^{*}q+Fh=A \geq 0$, and
\begin{equation}
\label{eq:nnl}
    \| \pi_{2}^{*}p+\pi_{1}^{*}q+Fh \|_{C^{\beta r}} \leq \| \pi_{2}^{*}p+\pi_{1}^{*}q \|_{C^{\beta}}. 
\end{equation}
Now, take $0 <s\ll r$ and use the stability of the normal operator $N_2$ over the solenoidal tensor pair $f = [p,q]$ to compute
\[
\begin{split}
    \|f\|_{H^{-1-s}} &\lesssim \|N_{2}f\|_{H^{-s}}
    \lesssim \| (\Pi+1 \otimes 1)(\pi_{2}^{*}p+\pi_{1}^{*}q) \|_{H^{-s}}
\end{split}
\]
where the second inequality follows from the definition of $N_2$ and the boundedness of $\pi_{1*}, \pi_{2*} \colon H^{-s}(SM) \to H^{-s}(M)$. Applying the triangle inequality to the above, then using $\Pi F = 0$ (Lemma \ref{lemma:Pi_prop} (1)) and that $F$ preserves the Liouville measure allows us to continue this inequality chain:
\begin{equation}
\begin{split}
\|f\|_{H^{-1-s}} 
&\lesssim \| \Pi(\pi_{2}^{*}p+\pi_{1}^{*}q) \|_{H^{-s}} + |\langle \pi_{2}^{*}p+\pi_{1}^{*}q,1 \rangle|_{L^{2}}\\
&\lesssim \| \Pi(\pi_{2}^{*}p+\pi_{1}^{*}q+Fh) \|_{H^{-s}} + |\langle \pi_{2}^{*}p+\pi_{1}^{*}q +Fh,1 \rangle|_{L^{2}}.
\end{split}
\label{eq:first-bound}
\end{equation}
Now we apply the boundedness of $\Pi \colon H^{s}(SM) \to H^{-s}(SM)$ (Lemma \ref{lemma:Pi_prop}), then Cauchy--Schwarz to further continue this inequality chain and conclude
\[
\begin{split}
    \|f\|_{H^{-1-s}} 
    &\lesssim \| \pi_{2}^{*}p+\pi_{1}^{*}q+Fh\|_{H^{s}} + |\langle \pi_{2}^{*}p+\pi_{1}^{*}q +Fh,1 \rangle|_{L^{2}} \\
    & \lesssim  \| \pi_{2}^{*}p+\pi_{1}^{*}q+Fh\|_{H^{s}} \\
    & \lesssim  \| \pi_{2}^{*}p+\pi_{1}^{*}q+Fh\|_{L^{2}}^{1-\nu} \| \pi_{2}^{*}p+\pi_{1}^{*}q+Fh\|_{H^{r'}}^{\nu}
\end{split}
\]
where in the last step we take $s < r' < \beta r$ and interpolate with $\nu=s/r'$.
Now we proceed to bound each term in the right-hand side. First, observe
\begin{equation} 
\begin{split}
    \| \pi_{2}^{*}p+\pi_{1}^{*}q+Fh\|_{H^{r'}} &\lesssim \|\pi_{2}^{*}p+\pi_{1}^{*}q+Fh\|_{C^{r'}} \\ 
    & \lesssim \|\pi_{2}^{*}p+\pi_{1}^{*}q+Fh\|_{C^{\beta r}} \\
    & \lesssim \|\pi_{2}^{*}p+\pi_{1}^{*}q\|_{C^{\beta}},
\end{split}
\label{eq:2nd-term}
\end{equation}
where we used \eqref{eq:nnl} in the third inequality. On the other hand, we can interpolate between Lebesgue spaces to estimate
\begin{equation}
\begin{split}
    \|\pi_{2}^{*}p+\pi_{1}^{*}q+Fh\|_{L^{2}} &\lesssim \| \pi_{2}^{*}p+\pi_{1}^{*}q+Fh\|_{L^{1}}^{1/2} \| \pi_{2}^{*}p+\pi_{1}^{*}q+Fh\|_{L^{\infty}}^{1/2} \\
    &\lesssim \| \pi_{2}^{*}p+\pi_{1}^{*}q\|_{L^{1}}^{1/2} \| \pi_{2}^{*}p+\pi_{1}^{*}q+Fh\|_{C^{\beta r}}^{1/2}
    \label{eq:1st-term-1}
\end{split}
\end{equation}
where in the second inequality, we used $\pi_{2}^{*}p+\pi_{1}^{*}q+Fh > 0$ and that $F$ preserves the Liouville measure to make the bound 
\[\|\pi_{2}^{*}p+\pi_{1}^{*}q+Fh\|_{L^1} = \ip{\pi_{2}^{*}p+\pi_{1}^{*}q+Fh}{1}_{L^2} = \ip{\pi_{2}^{*}p+\pi_{1}^{*}q}{1}_{L^2} \leq \|\pi_{2}^{*}p+\pi_{1}^{*}q\|_{L^1}.
\]
Continuing the inequality chain \eqref{eq:1st-term-1} by first making standard inclusions, then by applying the stability estimate from Lemma \ref{lemma:stab} (recall $D_{\mu}^{*}f=0$) we find
\begin{equation}
\begin{split}
    \|\pi_{2}^{*}p+\pi_{1}^{*}q+Fh\|_{L^{2}}
    &\lesssim \| \pi_{2}^{*}p+\pi_{1}^{*}q\|_{H^{s_{0}}}^{1/2} \| \pi_{2}^{*}p+\pi_{1}^{*}q\|_{C^{\beta}}^{1/2} \\
    & \lesssim \|I_{2}f\|_{\ell^{\infty}(\mathcal{C})}^{\tau/2}\|f\|_{C^{1}}^{(1-\tau)/2} \|f\|_{C^{\beta}}^{1/2} \\
    & \lesssim \|f\|_{C^{4}}^{\tau}\|f\|_{C^{1}}^{(1-\tau)/2} \|f\|_{C^{\beta}}^{1/2} \\
    & \lesssim \|f\|_{C^{4}}^{(1+\tau)/2} \|f\|_{C^{\beta}}^{1/2}
\end{split}
\label{eq:1st-term-2}
\end{equation}
where in the third step we used that $(g_{0},\alpha_{0})$ and $(g,\alpha)$ have the same marked magnetic action and Lemma \ref{lemma:taylor}. In summary, by applying  \eqref{eq:2nd-term} and \eqref{eq:1st-term-2} to the bound \eqref{eq:first-bound}, we have
\begin{equation} \label{eq:ineq_smooth_reg}
    \|f\|_{H^{-1-s}} \leq C \|f\|_{C^{4}}^{(1+\tau)(1-\nu)/2} \|f\|_{C^{\beta}}^{(1+\nu)/2},
\end{equation}
where $C=C(g_{0},\alpha_{0},s,r,n)$. Let $j \in \{\beta,4\}$, and let $N_{0} \geq n/2+j+s$. By Interpolation and the Sobolev embedding we have
\begin{equation} \label{eq:sobolev}
    \|f\|_{C^{j}} \leq\|f\|_{H^{n/2+j+s}} \leq C\|f\|_{H^{-1-s}}^{1-\theta_{j}}\|f\|_{H^{N_{0}}}^{\theta_{j}},
\end{equation}

where $\theta_{j}=(n/2+j+1+2s)/(N_{0}+s+1)$. If we take $0<s<r'<\beta r \ll \beta$ small enough, and $N_{0}$ big enough, we have
\begin{equation} \label{eq:gamma}
    \gamma:=\frac{1}{2}(1-\theta_{\beta})(1+\nu)+\frac{1}{2}(1+\tau)(1-\theta_{3})(1-\nu)>1.
\end{equation}
Write $\gamma':=(1+\nu)\theta_{\beta}/2 +(1-\nu)\theta_{3}(1+\tau)/2>0$. Then, \eqref{eq:sobolev} together with \eqref{eq:ineq_smooth_reg} give
\begin{equation} \label{eq:final_ineq}
    \|f\|_{H^{-1-s}} \leq C\|f\|_{H^{-1-s}}^{\gamma} \|f\|_{H^{N_{0}}}^{\gamma'}. 
\end{equation}
Now, assume for sake of contradiction that $f \neq 0$. Then, it follows from \eqref{eq:final_ineq} that
\[
1 \leq C \|f\|_{H^{N_{0}}}^{\gamma-1+\gamma'}.
\]
This gives a contradiction thanks to \eqref{eq:gamma}. Hence, it is enough to take $N \geq N_{0}$ to obtain the result.

\section{Proof of Theorem \ref{thm:stability}} \label{sec:proof_stability}

Let $f$ be as in the proof of Theorem \ref{thm:local_full}. Note that since $\phi$ is close to the identity, we can assume that $\|f\|_{C^{N}}<\delta$, where $\delta$ is a small constant. Then, it follows from Lemmas \ref{lemma:stab} and \ref{lemma:taylor} that for $s > 0$ small enough we have
\begin{align*}
    \| f\|_{H^{s}} &\lesssim \|I_{2}f \|_{\ell^{\infty}(\mathcal{C})}^{\tau} \|f\|_{C^{1}}^{1-\tau} \\
    &\lesssim (\|\A_{g,\alpha}-\A_{g_{0},\alpha_{0}}\|_{\ell^{\infty}(\mathcal{C})}+\|f\|_{C^{4}}^{2})^{\tau} \|f\|_{C^{1}}^{1-\tau} \\
    & \lesssim \|\A_{g,\alpha}-\A_{g_{0},\alpha_{0}}\|_{\ell^{\infty}(\mathcal{C})}^{\tau}\|f\|_{C^{1}}^{1-\tau}+ \|f\|_{C^{4}}^{2\tau}\|f\|_{C^{1}}^{1-\tau}.
\end{align*}
As in the proof of \cite{2019-guillarmou-lefeuvre}*{Theorem 3}, we will use interpolation to absorb the last term on the right-hand side on the left-hand side. To do so, let $j \in\{1,4\}$. In virtue of the Sobolev embedding $H^{n/2+j+s} \hookrightarrow C^{j}$ and interpolation between $H^{s}$ and $H^{N_{0}}$ (where $N_{0} > n/2+j+s$) we obtain
\[
\|f\|_{C^{j}} \lesssim \|f\|_{H^{n/2+j+s}} \lesssim \|f\|_{H^{s}}^{1-\theta_{j}} \|f\|_{H^{N_{0}}}^{\theta_{j}},
\]
where $\theta_{j}=(n/2+j)/(N_{0}-s)$. Hence, writing
\[  
\gamma=(1-\theta_{4})2\tau+(1-\theta_{1})(1-\tau), \qquad \gamma'=2\tau \theta_{4}+(1-\tau)\theta_{1},
\]
we have that for small $s$ and $N_{0}$ big enough, $\gamma>1$ and $\gamma'>0$. Hence,
\[
\|f\|_{C^{4}}^{2\tau}\|f\|_{C^{1}}^{1-\tau} \lesssim \|f\|_{H^{s}}^{\gamma}\|f\|_{H^{N_{0}}}^{\gamma'} \lesssim \|f\|_{H^{s}}^{\gamma}\|f\|_{C^{N_{0}}}^{\gamma'} \lesssim \|f\|_{H^{s}}\|f\|_{C^{N_{0}}}^{(\gamma-1)+\gamma'}.
\]
So far, we have obtained that for $N>N_{0}$, there exist $C>0$ with
\[ 
\|f\|_{H^{s}} \leq C(\|\A_{g,\alpha}-\A_{g_{0},\alpha_{0}}\|_{\ell^{\infty}(\mathcal{C})}^{\tau}\|f\|_{C^{1}}^{1-\tau}+\|f\|_{H^{s}}\|f\|_{C^{N_{0}}}^{(\gamma-1)+\gamma'}).
\]
Taking $\delta$ small enough so that $C\delta^{\gamma-1+\gamma'} \leq 1/2$, we obtain the claimed absorption. The result now follows from the bounds presented here and in Lemma \ref{lemma:slice}.

\section{Proof of Theorem \ref{thm:conformal}} \label{section:proof_conformal}

Let us begin by assuming without loss of generality, that $\vol(g_{2}) \leq \vol (g_{1}) = 1$. Then, 
\begin{equation} \label{eq:energy}
    \int_{M}e^{2f}d\vol_{g_{1}} \leq \left( \int_{M}e^{nf} d\vol_{g_{1}}\right)^{\frac{2}{n}} \vol(g_{1})^{\frac{n-2}{n}}
    \leq  \vol(g_{2})^{\frac{2}{n}} \leq 1,
\end{equation}
where in the first step we used H\"older's inequality, and in the second and third steps we used the normalization assumption on the volume. In a similar fashion, we have
\begin{equation}\label{eq:length}
    \int_{M}e^{f}d\vol_{g_{1}} \leq \left( \int_{M}e^{nf} d\vol_{g_{1}} \right)^{\frac{1}{n}} \vol(g_{1})^{\frac{n-1}{n}} \leq\vol(g_{2})^{\frac{1}{n}} \leq 1.
\end{equation}
After adding \eqref{eq:energy} and \eqref{eq:length}, for sake of contradiction, assume that the inequality is strict, i.e., there exists an $\eps>0$ with 
\begin{equation} \label{eq:ineq-contradiction-conformal}
    \frac{1}{2}\left(\int_{M}e^{2f}d \vol_{g_{1}}+\int_{M}e^{f}d\vol_{g_{1}} \right) <1-\eps.
\end{equation}
Let $\mu_{g_{1}}$ be the normalized Liouville measure. Then, as a consequence of Birkhoff's Ergodic Theorem (see for instance \cite{Lefeuvre_book}*{Lemma 8.4.9}), there exists a sequence of closed magnetic geodesics $(\gamma_{g_{1},\alpha_{1}}(c_{j}))_{j \geq 0}$ such that $\delta_{\gamma_{g_{1},\alpha_{1}}(c_{n})}$ converges to $\mu_{g_{1}}$ in the weak-$\ast$ topology. Let us denote the energy of a curve $\gamma$ (with respect to a metric $g$) by $E_{g}(\gamma)$. Then, for $j$ big enough, we have
\[ 
\begin{split}
    \frac{E_{g_{2}}(\gamma_{g_{1},\alpha_{1}}(c_{j})) + \frac{1}{2} \ell_{g_{2}}(\gamma_{g_{1},\alpha_{1}}(c_{j}))}{\ell_{g_{1}}(\gamma_{g_{1},\alpha_{1}}(c_{j}))}
    &=\frac{1}{2\ell_{g_{1}}(\gamma_{g_{1},\alpha_{1}}(c_{j}))} \int_{\gamma_{g_{1},\alpha_{1}}(c_{j})}(e^{2f}+e^{f}) \\
    &< \frac{1}{2} \int_{SM}(e^{2\pi_{0}^{*}f}+e^{\pi_{0}^{*}f})d\mu_{g_{1}}+\frac{\eps}{2} \\
    &= \frac{1}{2} \int_{M}(e^{2f}+e^{f})d\vol_{g_{1}}+\frac{\eps}{2} \\
    &<1-\frac{\eps}{2},
\end{split}
\]
where in the first step we used that $g_{1}$ and $g_{2}$ are conformal, in the second one the denseness of periodic magnetic geodesics, while the last one follows from \eqref{eq:ineq-contradiction-conformal}. Now observe that since $\pi_{1}^{*}\alpha_{1}$ and $\pi_{1}^{*}\alpha_{2}$ are odd functions, we have 
\begin{equation} \label{eq:int-alpha}
    \int_{\gamma_{g_{1},\alpha_{1}}(c_{j})} \alpha_{1}, \int_{\gamma_{g_{1},\alpha_{1}}(c_{j})} \alpha_{2} \xrightarrow{j \to \infty} 0. 
\end{equation}
Then, we have for $j$ big enough
\[ 
\begin{split}
    1-\frac{\eps}{2}  > & \frac{1}{\ell_{g_{1}}(\gamma_{g_{1},\alpha_{1}}(c_{j}))}\left( \mathbb{A}_{g_{2},\alpha_{2}}(\gamma_{g_{1},\alpha_{1}}(c_{j}))+\int_{\gamma_{g_{1},\alpha_{1}}(c_{j})}\alpha_{2} \right) \\
     \geq & \frac{1}{\ell_{g_{1}}(\gamma_{g_{1},\alpha_{1}}(c_{j}))} \mathbb{A}_{g_{2},\alpha_{2}}(\gamma_{g_{2},\alpha_{2}}(c_{j}))+\frac{1}{\ell_{g_{1}}(\gamma_{g_{1},\alpha_{1}}(c_{j}))} \int_{\gamma_{g_{1},\alpha_{1}}(c_{j})}\alpha_{2}  \\
     = & \left( \A_{g_{1},\alpha_{1}}(\gamma_{g_{1},\alpha_{1}}(c_{j}))+\int_{\gamma_{g_{1},\alpha_{1}}(c_{j})}\alpha_{1} \right)^{-1} \mathbb{A}_{g_{2},\alpha_{2}}(\gamma_{g_{2},\alpha_{2}}(c_{j})) \\
     & +\frac{1}{\ell_{g_{1}}(\gamma_{g_{1},\alpha_{1}}(c_{j}))} \int_{\gamma_{g_{1},\alpha_{1}}(c_{j})}\alpha_{2}  \\
     = & \left( \A_{g_{1},\alpha_{1}}(\gamma_{g_{1},\alpha_{1}}(c_{j}))+\int_{\gamma_{g_{1},\alpha_{1}}(c_{j})}\alpha_{1} \right)^{-1} \mathbb{A}_{g_{1},\alpha_{1}}(\gamma_{g_{1},\alpha_{1}}(c_{j})) \\
     & +\frac{1}{\ell_{g_{1}}(\gamma_{g_{1},\alpha_{1}}(c_{j}))} \int_{\gamma_{g_{1},\alpha_{1}}(c_{j})}\alpha_{2} \\
     = &  \left( \frac{\A_{g_{1},\alpha_{1}}(\gamma_{g_{1},\alpha_{1}}(c_{j}))}{\ell_{g_{1}}(\gamma_{g_{1},\alpha_{1}}(c_{j}))}+\frac{1}{\ell_{g_{1}}(\gamma_{g_{1},\alpha_{1}}(c_{j}))}\int_{\gamma_{g_{1},\alpha_{1}}(c_{j})}\alpha_{1} \right)^{-1} \frac{\mathbb{A}_{g_{1},\alpha_{1}}(\gamma_{g_{1},\alpha_{1}}(c_{j}))}{\ell_{g_{1}}(\gamma_{g_{1},\alpha_{1}}(c_{j}))} \\
     &+\frac{1}{\ell_{g_{1}}(\gamma_{g_{1},\alpha_{1}}(c_{j}))} \int_{\gamma_{g_{1},\alpha_{1}}(c_{j})}\alpha_{2}, \\
\end{split}
\]
where in the first step we used that magnetic geodesic minimize the action, in the second we used the definition of the action, and the third follows from the equality of the actions. Thus, letting $j \to \infty$, and using \eqref{eq:int-alpha}, we obtain a contradiction. This shows that $f \equiv 0$. To show the relations between the magnetic parts, we observe that 
\[ 
\begin{split}
    \ell_{g_{1}}(\gamma_{g_{1},\alpha_{1}}(c))-\int_{\gamma_{g_{1},\alpha_{1}}(c)}\alpha_{1} & = \mathbb{A}_{g_{1},\alpha_{1}}(\gamma_{g_{1},\alpha_{1}}(c)) \\
    & = \mathbb{A}_{g_{2},\alpha_{2}}(\gamma_{g_{2},\alpha_{2}}(c)) \\
    & \leq \mathbb{A}_{g_{2},\alpha_{2}}(\gamma_{g_{1},\alpha_{1}}(c)) \\
    & =\ell_{g_{1}}(\gamma_{g_{1},\alpha_{1}}(c))-\int_{\gamma_{g_{1},\alpha_{1}}(c)} \alpha_{2},
\end{split}
\]
we in the second step we used the equality of the marked actions, in the third one we used that the action is minized over magnetic geodesics, and the last one follows from the definition of the action and the fact that $g_{1}=g_{2}$. Then,
\begin{equation} \label{eq:int-alpha-non-negative}
    \int_{\gamma_{g_{1},\alpha_{1}}(c)}(\alpha_{1}-\alpha_{2}) \geq 0.
\end{equation}
We can invoke the non-negative Liv\v sic theorem (Lemma \ref{lemma:non-negative_livsic}) to obtain a H\"older function $U$ so that for any $x$ and $s$ we have 
\begin{equation} \label{eq:nnl-magnetic}
    \int_{0}^{s} \pi_{1}^{*}\alpha_{1}- \pi_{1}^{*}\alpha_{2} + U(\varphi_{s}x)-U(x) \geq 0.
\end{equation}
Using again that the integral of 1-forms over $SM$ is zero, and the fact the magnetic flow preserves the Liouville measure, we obtain equality in \eqref{eq:nnl-magnetic}. In particular, taking $s$ as the period of the closed curves, we obtain equality in \eqref{eq:int-alpha-non-negative}. Hence, the results follows from tensor tomography for magnetic flows for 1-tensors \cite{DP08}*{Theorem B}.


\bibliographystyle{alpha}
\bibliography{bib}

\begin{bibdiv}
\begin{biblist}

\bib{AdSMRT-magnetic-rigidity-2d}{article}{
      author={Assenza, Valerio},
      author={de~Simoi, Jacopo},
      author={Marshall~Reber, James},
      author={Terek, Ivo},
       title={Marked length spectrum rigidity for anosov magnetic surfaces},
        date={2024},
     journal={arXiv preprint arXiv:2409.20545},
}

\bib{Ainsworth15}{article}{
      author={Ainsworth, Gareth},
       title={The magnetic ray transform on anosov surfaces},
        date={2015},
        ISSN={1078-0947},
     journal={Discrete Contin. Dyn. Syst.},
      volume={35},
      number={5},
       pages={1801\ndash 1816},
      review={\MR{3294225}},
}

\bib{Assenza24}{article}{
      author={Assenza, Valerio},
       title={Magnetic curvature and existence of a closed magnetic geodesic on
  low energy levels},
        date={2024},
        ISSN={1073-7928},
     journal={Int. Math. Res. Not. IMRN},
      number={21},
       pages={13586\ndash 13610},
      review={\MR{4819867}},
}

\bib{Beaufort-magnetic-tom}{article}{
      author={Beaufort, Louis-Brahim},
       title={Tensor tomography and frame flow ergodicity for magnetic flows in
  higher dimensions},
        date={2026},
     journal={In preparation},
}

\bib{1985-burns-katok}{article}{
      author={Burns, K.},
      author={Katok, A.},
       title={Manifolds with nonpositive curvature},
        date={1985},
        ISSN={0143-3857,1469-4417},
     journal={Ergodic Theory Dynam. Systems},
      volume={5},
      number={2},
       pages={307\ndash 317},
         url={https://doi.org/10.1017/S0143385700002935},
      review={\MR{796758}},
}

\bib{BK02}{article}{
      author={Burns, Keith},
      author={Paternain, Gabriel~P.},
       title={Anosov magnetic flows, critical values and topological entropy},
        date={2002},
        ISSN={0951-7715,1361-6544},
     journal={Nonlinearity},
      volume={15},
      number={2},
       pages={281\ndash 314},
  url={https://doi-org.offcampus.lib.washington.edu/10.1088/0951-7715/15/2/305},
      review={\MR{1888853}},
}

\bib{Butt_approx_mls}{article}{
      author={Butt, Karen},
       title={Approximate rigidity of the marked length spectrum},
        date={2023},
        ISSN={2772-9559},
     journal={Math. Res. Rep.},
      volume={4},
       pages={63\ndash 82},
         url={https://doi.org/10.5802/mrr.18},
      review={\MR{4687792}},
}

\bib{Butt_quantitative_mls}{article}{
      author={Butt, Karen},
       title={Quantitative marked length spectrum rigidity},
        date={2025},
        ISSN={1465-3060,1364-0380},
     journal={Geom. Topol.},
      volume={29},
      number={8},
       pages={3995\ndash 4054},
         url={https://doi.org/10.2140/gt.2025.29.3995},
      review={\MR{4999725}},
}

\bib{CIPP00}{article}{
      author={Contreras, G.},
      author={Iturriaga, R.},
      author={Paternain, G.~P.},
      author={Paternain, M.},
       title={The palais-smale condition and ma\~n\'e's critical values},
        date={2000},
        ISSN={1424-0637},
     journal={Ann. Henri Poincar\'e},
      volume={1},
      number={4},
       pages={655\ndash 684},
      review={\MR{1785184}},
}

\bib{Croke-rigidity-2d}{article}{
      author={Croke, Christopher~B.},
       title={Rigidity for surfaces of nonpositive curvature},
        date={1990},
        ISSN={0010-2571,1420-8946},
     journal={Comment. Math. Helv.},
      volume={65},
      number={1},
       pages={150\ndash 169},
         url={https://doi.org/10.1007/BF02566599},
      review={\MR{1036134}},
}

\bib{CS98}{article}{
      author={Croke, Christopher~B.},
      author={Sharafutdinov, Vladimir~A.},
       title={Spectral rigidity of a compact negatively curved manifold},
        date={1998},
        ISSN={0040-9383},
     journal={Topology},
      volume={37},
      number={6},
       pages={1265\ndash 1273},
         url={https://doi.org/10.1016/S0040-9383(97)00086-4},
      review={\MR{1632920}},
}

\bib{de_la_Llave_Marco_Moriyon_perturbation}{article}{
      author={de~la Llave, R.},
      author={Marco, J.~M.},
      author={Moriy\'on, R.},
       title={Canonical perturbation theory of {A}nosov systems and regularity
  results for the {L}iv\v sic cohomology equation},
        date={1986},
        ISSN={0003-486X,1939-8980},
     journal={Ann. of Math. (2)},
      volume={123},
      number={3},
       pages={537\ndash 611},
         url={https://doi.org/10.2307/1971334},
      review={\MR{840722}},
}

\bib{DP05}{article}{
      author={Dairbekov, Nurlan~S.},
      author={Paternain, Gabriel~P.},
       title={Longitudinal {KAM}-cocycles and action spectra of magnetic
  flows},
        date={2005},
        ISSN={1073-2780},
     journal={Math. Res. Lett.},
      volume={12},
      number={5-6},
       pages={719\ndash 729},
         url={https://doi.org/10.4310/MRL.2005.v12.n5.a9},
      review={\MR{2189233}},
}

\bib{DP07}{article}{
      author={Dairbekov, Nurlan~S.},
      author={Paternain, Gabriel~P.},
       title={Entropy production in gaussian thermostats},
        date={2007},
        ISSN={0010-3616},
     journal={Comm. Math. Phys.},
      volume={269},
      number={2},
       pages={533\ndash 543},
      review={\MR{2274556}},
}

\bib{DP08}{article}{
      author={Dairbekov, Nurlan~S.},
      author={Paternain, Gabriel~P.},
       title={Rigidity properties of {A}nosov optical hypersurfaces},
        date={2008},
        ISSN={0143-3857},
     journal={Ergodic Theory Dynam. Systems},
      volume={28},
      number={3},
       pages={707\ndash 737},
      review={\MR{2422013}},
}

\bib{DPSU07}{article}{
      author={Dairbekov, Nurlan~S.},
      author={Paternain, Gabriel~P.},
      author={Stefanov, Plamen},
      author={Uhlmann, Gunther},
       title={The boundary rigidity problem in the presence of a magnetic
  field},
        date={2007},
        ISSN={0001-8708},
     journal={Adv. Math.},
      volume={216},
      number={2},
       pages={535\ndash 609},
      review={\MR{2351370}},
}

\bib{Ebin_slice}{article}{
      author={Ebin, David~G.},
       title={On the space of {R}iemannian metrics},
        date={1968},
        ISSN={0002-9904},
     journal={Bull. Amer. Math. Soc.},
      volume={74},
       pages={1001\ndash 1003},
         url={https://doi.org/10.1090/S0002-9904-1968-12115-9},
      review={\MR{231410}},
}

\bib{EC25}{article}{
      author={Echevarr\'ia~Cuesta, Javier},
       title={Smooth orbit equivalence rigidity for dissipative geodesic
  flows},
        date={2025},
        ISSN={0143-3857,1469-4417},
     journal={Ergodic Theory Dynam. Systems},
      volume={45},
      number={12},
       pages={3698\ndash 3727},
         url={https://doi.org/10.1017/etds.2025.20},
      review={\MR{4986555}},
}

\bib{2011-faure-sjostrand}{article}{
      author={Faure, Fr\'ed\'eric},
      author={Sj\"ostrand, Johannes},
       title={Upper bound on the density of {R}uelle resonances for {A}nosov
  flows},
        date={2011},
        ISSN={0010-3616,1432-0916},
     journal={Comm. Math. Phys.},
      volume={308},
      number={2},
       pages={325\ndash 364},
         url={https://doi.org/10.1007/s00220-011-1349-z},
      review={\MR{2851145}},
}

\bib{GKL22}{article}{
      author={Guillarmou, Colin},
      author={Knieper, Gerhard},
      author={Lefeuvre, Thibault},
       title={Geodesic stretch, pressure metric and marked length spectrum
  rigidity},
        date={2022},
        ISSN={0143-3857},
     journal={Ergodic Theory Dynam. Systems},
      volume={42},
      number={3},
       pages={974\ndash 1022},
      review={\MR{4374964}},
}

\bib{2019-guillarmou-lefeuvre}{article}{
      author={Guillarmou, Colin},
      author={Lefeuvre, Thibault},
       title={The marked length spectrum of {A}nosov manifolds},
        date={2019},
        ISSN={0003-486X,1939-8980},
     journal={Ann. of Math. (2)},
      volume={190},
      number={1},
       pages={321\ndash 344},
         url={https://doi.org/10.4007/annals.2019.190.1.6},
      review={\MR{3990606}},
}

\bib{GLP25}{article}{
      author={Guillarmou, Colin},
      author={Lefeuvre, Thibault},
      author={Paternain, Gabriel~P.},
       title={Marked length spectrum rigidity for {A}nosov surfaces},
        date={2025},
        ISSN={0012-7094,1547-7398},
     journal={Duke Math. J.},
      volume={174},
      number={1},
       pages={131\ndash 157},
         url={https://doi.org/10.1215/00127094-2024-0022},
      review={\MR{4866545}},
}

\bib{Guillarmou17}{article}{
      author={Guillarmou, Colin},
       title={Invariant distributions and {X}-ray transform for {A}nosov
  flows},
        date={2017},
        ISSN={0022-040X},
     journal={J. Differential Geom.},
      volume={105},
      number={2},
       pages={177\ndash 208},
      review={\MR{3606728}},
}

\bib{Hamenstadt-rigidity-LS}{article}{
      author={Hamenst\"adt, U.},
       title={Cocycles, symplectic structures and intersection},
        date={1999},
        ISSN={1016-443X,1420-8970},
     journal={Geom. Funct. Anal.},
      volume={9},
      number={1},
       pages={90\ndash 140},
         url={https://doi.org/10.1007/s000390050082},
      review={\MR{1675892}},
}

\bib{Katok_conformal}{article}{
      author={Katok, Anatole},
       title={Four applications of conformal equivalence to geometry and
  dynamics},
        date={1988},
        ISSN={0143-3857,1469-4417},
     journal={Ergodic Theory Dynam. Systems},
      volume={8$^*$},
       pages={139\ndash 152},
         url={https://doi.org/10.1017/S0143385700009391},
      review={\MR{967635}},
}

\bib{Lefeuvre_book}{book}{
      author={Lefeuvre, Thibault},
       title={Microlocal analysis in hyperbolic dynamics and geometry},
      series={Cours Sp\'ecialis\'es [Specialized Courses]},
   publisher={Soci\'et\'e{} Math\'ematique de France, Paris},
        date={2025},
      volume={32},
        ISBN={978-2-37905-221-7},
      review={\MR{5022124}},
}

\bib{Lopes_Thieullen_Livsic}{article}{
      author={Lopes, A.~O.},
      author={Thieullen, Ph.},
       title={Sub-actions for {A}nosov flows},
        date={2005},
        ISSN={0143-3857,1469-4417},
     journal={Ergodic Theory Dynam. Systems},
      volume={25},
      number={2},
       pages={605\ndash 628},
         url={https://doi.org/10.1017/S0143385704000732},
      review={\MR{2129112}},
}

\bib{M97}{article}{
      author={Ma\~n\'e, Ricardo},
       title={Lagrangian flows: the dynamics of globally minimizing orbits},
        date={1997},
        ISSN={0100-3569},
     journal={Bol. Soc. Brasil. Mat. (N.S.)},
      volume={28},
      number={2},
       pages={141\ndash 153},
         url={https://doi-org.offcampus.lib.washington.edu/10.1007/BF01233389},
      review={\MR{1479499}},
}

\bib{2025-munoz-thon_Richardson}{article}{
      author={Mu\~noz Thon, Sebasti\'an},
      author={Richardon, Sean},
       title={Guillarmou's normal operator for magnetic and thermostat flows},
        date={2025},
     journal={arXiv preprint arXiv:2512.23106},
}

\bib{MR-deformative}{article}{
      author={Marshall~Reber, James},
       title={Deformative magnetic marked length spectrum rigidity},
        date={2023},
        ISSN={0024-6093,1469-2120},
     journal={Bull. Lond. Math. Soc.},
      volume={55},
      number={6},
       pages={3077\ndash 3096},
         url={https://doi.org/10.1112/blms.12911},
      review={\MR{4698329}},
}

\bib{Otal-rigidity-2d}{article}{
      author={Otal, Jean-Pierre},
       title={Le spectre marqu\'e{} des longueurs des surfaces \`a{} courbure
  n\'egative},
        date={1990},
        ISSN={0003-486X,1939-8980},
     journal={Ann. of Math. (2)},
      volume={131},
      number={1},
       pages={151\ndash 162},
         url={https://doi.org/10.2307/1971511},
      review={\MR{1038361}},
}

\end{biblist}
\end{bibdiv}

\end{document}